\documentclass[10pt]{amsart}
\title{Rigidity of almost-isometric universal covers}

\author{Aditi Kar}
\address{Mathematical Institute\\
University of Oxford\\
Oxford\\
OX2 6GG\\
United Kingdom}

\author{Jean-Fran\c{c}ois Lafont}
\address{Department of Mathematics\\
                 Ohio State University\\
                 Columbus, Ohio 43210}

\author{Benjamin Schmidt}
\address{Department of Mathematics \\ 
                 Michigan State University \\
                 East Lansing, MI 48824}

\date{\today}

\theoremstyle{theorem}
\newtheorem{thm}{Theorem}[section]
\newtheorem{lem}[thm]{Lemma}
\newtheorem{cor}[thm]{Corollary}
\newtheorem{prop}[thm]{Proposition}

\theoremstyle{definition}
\newtheorem{rem}{Remark}
\newtheorem{exmp}[thm]{Example}
\newtheorem{ques}{Question}
\newtheorem{defn}[thm]{Definition}

\numberwithin{equation}{section}

\def\Pb{\ifmmode{\Bbb P}\else{$\Bbb P$}\fi}
\def\Z{\ifmmode{\Bbb Z}\else{$\Bbb Z$}\fi}
\def\Q{\ifmmode{\Bbb Q}\else{$\Bbb Q$}\fi}
\def\C{\ifmmode{\Bbb C}\else{$\Bbb C$}\fi}
\def\R{\ifmmode{\Bbb R}\else{$\Bbb R$}\fi}
\def\H{\ifmmode{\Bbb H}\else{$\Bbb H$}\fi}
\def\S{\ifmmode{S^2}\else{$S^2$}\fi}

\def\spec{\operatorname{Spec}}

\def\diam{\operatorname{diam}}

\def\S{\mathcal S}

\def\Isom{\operatorname{Isom}}
\def\Id{\operatorname{Id}}
\def\Min{\operatorname{Min}}

\usepackage{graphicx} 

\usepackage{amsmath,amssymb,amsthm,amsfonts,amscd,flafter,epsf}
\usepackage{graphics}
\usepackage{epsfig}
\usepackage{psfrag}
\usepackage[all]{xy}

\begin{document}

\begin{abstract}
Almost-isometries are quasi-isometries with multiplicative constant one.  Lifting a pair of metrics on a compact space gives quasi-isometric metrics on the universal cover. Under some 
additional hypotheses on the metrics, we show that there is no almost-isometry between the universal covers.  We show that Riemannian manifolds which are almost-isometric have the same volume growth entropy. We establish various rigidity results as applications.
\end{abstract}
\maketitle

\setcounter{secnumdepth}{1}

\setcounter{section}{0}

\section{\bf Introduction}

{\it Quasi-isometries} are the natural morphisms in asymptotic geometry. Their definition involves both an additive constant $\geq 0$ and a multiplicative constant $\geq 1$.  {\it Bi-Lipschitz maps} are quasi-isometries with additive constant equal to zero; we define {\it almost-isometries} as quasi-isometries with multiplicative constant equal to one.  When looking at a general inequality, 
it is often important to understand the equality cases. Thus a natural problem is to identify conditions which force quasi-isometric spaces to be either bi-Lipschitz equivalent, or almost-isometric.


For discrete spaces, bi-Lipschitz maps coincide with \textit{bijective} quasi-isometries.  In \cite{Wh}, quasi-isometry classes of maps that contain a bijective quasi-isometry are characterized for uniformly discrete metric spaces of bounded geometry.  In particular, quasi-isometric finitely generated groups that are nonamenable are bi-Lipschitz equivalent. In contrast, there exist separated nets in $\mathbb{R}^2$ that are quasi-isometric but not bi-Lipschitz equivalent to $\mathbb{Z}^2$ \cite{BK,Mc}.  Quasi-isometric finitely generated groups that are not bi-Lipschitz equivalent first appeared in \cite{Dy}. While those first examples were not finitely presented, examples of type $F_n$ for each $n$ appear in \cite{DPT}.

Existing results about almost-isometries primarily concern equivalence classes of metrics on a fixed space, where two metrics are equivalent when the \textit{identity map} is an almost-isometry.  For instance, pairs of $\mathbb{Z}^n$-equivariant metrics on $\mathbb{R}^n$ whose ratio tends to one as distances tend to infinity are equivalent by \cite{Bu}. Analogous results hold for metrics periodic under Gromov hyperbolic and Heisenberg groups \cite{Kr} or under toral relatively hyperbolic groups \cite{Fuj}.  The equivalence classes of left-invariant metrics on non-elementary Gromov hyperbolic groups are studied in \cite{Fur} where the \textit{Marked Length Spectrum (MLS) Rigidity Conjecture} is reformulated as follows.

\textit{Given negatively curved Riemannian metrics $g_0$ and $g_1$ on a compact manifold $M$, the identity map $(\tilde{M},\tilde g_0) \rightarrow (\tilde{M}, \tilde g_1)$ between the universal Riemannian coverings is an almost-isometry if and only if $(\tilde M,\tilde{g}_0)$ and $(\tilde M,\tilde{g}_1)$ are isometric}.  

In view of the resolution of the MLS Conjecture in dimension two \cite{Cr, Ot}, and expected validity in higher dimensions, the following question is quite natural.



\begin{ques}
If $g_0$ and $g_1$ are two negatively curved Riemannian metrics on a compact manifold $M$, can the Riemannian universal coverings $(\tilde{M}, \tilde g_0)$ and $(\tilde{M}, \tilde g_1)$ be almost-isometric without being isometric?
\end{ques}

In the above question, we allow non-identity and non-equivariant almost-isometries (thus generalizing Furman's reformulation
of the MLS conjecture).  Our focus in this paper is to show that under suitable rigidity hypotheses on the metrics $g_i$, the answer is ``no'' -- though in general the answer is ``yes'' (see \cite{LSvL} for some $2$-dimensional examples).

\begin{thm}\label{theorem-AI+MLS}
Let $G$ be a group acting geometrically on a proper CAT(-1) space $X$ (distinct from $\mathbb{R}$), and let $Y$ be another proper CAT(-1) space having the geodesic extension property and connected spaces of directions. Assume that:
\begin{itemize}
\item $Y$ is almost-isometrically rigid, and
\item the $G$-action on $X$ is marked length spectrum rigid.
\end{itemize}
Then $X$ and $Y$ are almost-isometric if and only if there is a coarsely onto isometric embedding of $X$ into $Y$.
\end{thm}

The rigidity properties required of the space $Y$ and the $G$-space $X$ are defined in Section \ref{section-preliminaries}.  Note that we are not assuming a $G$-action on $Y$ (in particular, there is no equivariance assumption on the almost-isometry).  As a concrete application of the methods behind Theorem \ref{theorem-AI+MLS}, we mention the following:


\begin{cor}\label{symmetric}
Let $(M, g_0)$ be a closed locally symmetric space modeled on quaternionic hyperbolic space, or on the Cayley hyperbolic plane, and let $g_1$ be a negatively curved Riemannian metric on $M$.  Then $(\tilde M, \tilde g_0)$ and $(\tilde M, \tilde g_1)$ are almost-isometric if and only if $(M,g_0)$ and $(M,g_1)$ are isometric.
\end{cor}

The ideas behind Theorem \ref{theorem-AI+MLS} also yield some rigidity results for Fuchsian buildings (see Corollary \ref{Fuchsian-building}).  After discussing some preliminaries in Section \ref{section-preliminaries}, we prove Theorem \ref{theorem-AI+MLS} and its Corollaries in Section \ref{section-proof-of-main}. 

In Section \ref{section-proof-of-volgrowth}, we relate the presence of almost-isometries 
with dynamical invariants. Recall that the {\it upper volume entropy} of a complete Riemannian manifold $X$ is defined to be
$$h_{vol}^+(X):=\limsup _{r\to \infty} \frac{\ln Vol(B_p(r))}{r}$$
where $B_p(r)$ denotes the ball of radius $r$ centered at a chosen fixed basepoint $p$.  Similarly, the {\it lower volume entropy}
is defined to be 
$$h_{vol}^-(X):=\liminf _{r\to \infty} \frac{\ln Vol(B_p(r))}{r}.$$
These quantities are independent of the chosen point $p \in X$, and in the case where $X$ is a Riemannian cover of a 
compact manifold, one has that $h_{vol}^+(X) = h_{vol}^-(X)$ (see \cite{Ma}); this common value is then called the 
{\it volume entropy} of $X$, and is denoted $h_{vol}(X)$.  In general, the upper and lower volume entropies can differ, 
even for Riemannian covers of finite volume manifolds \cite{Na}.




\begin{thm}\label{theorem-volgrowth}
Let $M_1, M_2$ be complete Riemannian manifolds having bounded sectional curvatures.  If $M_1$ is almost-isometric to $M_2$, then $h_{vol}^+(M_1)=h_{vol}^+(M_2)$, and $h_{vol}^-(M_1)=h_{vol}^-(M_2)$. In particular, if the $M_i$ are Riemannian covers of compact
manifolds, then $h_{vol}(M_1)=h_{vol}(M_2)$.
\end{thm}

In fact, the proof of Theorem \ref{theorem-volgrowth} only uses the property that $r$-balls in $M_i$ have volume uniformly bounded above and below by positive constants.  This property is a consequence of having bounded sectional curvatures by \cite{Bi,Gun}.


\begin{cor}\label{no-scale}
Let $(M, g)$ be a Riemannian cover of a compact manifold. If $h_{vol}(M)>0$, then for any positive $\lambda \neq 1$, the manifolds $(M, g)$ and $(M, \lambda g)$ are {\bf not} almost-isometric.
\end{cor}

For a closed Riemannian manifold $(M,g)$, the volume growth entropy $h_{vol}$ of its universal covering and the topological entropy $h_{top}$ of its geodesic flow satisfy $h_{vol} \leq h_{top}$ \cite{Ma}.  Equality holds for metrics without conjugate points \cite{FM},  a class of metrics including the nonpositively curved metrics, but is in general strictly larger \cite{Gul}.

\begin{cor}\label{entropy}
Let $g_0$ and $g_1$ be conjugate point free Riemannian metrics on a closed manifold $M$.  If the universal coverings 
$(\tilde M, \tilde g_0)$ and $(\tilde M, \tilde g_1)$ are almost-isometric, then $h_{top}(g_0) = h_{top}(g_1)$.
\end{cor}

As a final application of Theorem \ref{theorem-volgrowth}, we mention the following:

\begin{cor}\label{rigidity-results}
Let $M$ be a closed $n$-manifold equipped with Riemannian metrics $g_0$ and $g_1$ for which the universal coverings $(\tilde M, \tilde g_0)$ and $(\tilde M, \tilde g_1)$ are almost-isometric. Further assume that the metrics satisfy any 
of the following conditions:
\begin{enumerate}
\item $n=2$, $g_0$ is a flat metric, and $g_1$ is arbitrary, or
\item $n=2$, $g_0$ is a real hyperbolic metric, and $g_1$ satisfies $Vol(g_0)\geq Vol(g_1)$, or
\item $n\geq 3$, $g_0$ is a negatively curved locally symmetric metric, and $g_1$ satisfies $Vol(g_0)\geq Vol(g_1)$, or
\item $n\geq 5$, $g_0$ is an irreducible, higher rank, nonpositively curved locally symmetric metric, and $g_1$ is conformal 
to $g_0$ and satisfies $Vol(g_0) \geq Vol(g_1)$,
\item $n\geq 6$, $g_0$ is a locally symmetric metric modeled on a product of negatively curved symmetric spaces of dimension
$\geq 3$ (suitably normalized), and $g_1$ is any metric satisfying $Vol(g_0)\geq Vol(g_1)$.
\end{enumerate}
Then the universal covers $(\tilde M, \tilde g_0)$ and $(\tilde M, \tilde g_1)$ are isometric. In particular, $(M,g_0)$ is isometric to $(M, g_1)$ by Mostow rigidity in cases (3) - (5). 
\end{cor}

While the rigidity results Corollary \ref{symmetric} and Corollary \ref{rigidity-results} (3) both apply to locally symmetric metrics $g_0$ modeled on quaternionic hyperbolic space or on the Cayley hyperbolic plane, the former requires the metric $g_1$ to be negatively curved while the latter requires $g_1$ to have volume majorized by that of $g_0$.  

A discussion of rigidity results for almost-isometries between metric trees (Theorem \ref{general-volume-growth} and Corollary \ref{rigidity-tree}) appears at the end of Section \ref{section-proof-of-volgrowth}. Section \ref{section-concluding-remarks} concludes the paper with some remarks and open questions. 


\section{\bf Preliminaries}\label{section-preliminaries}

Throughout, parentheses are suppressed according to the following notational convention.  Given a function 
$\phi:X \rightarrow Y$ between sets and an element $x \in X$, the image $\phi (x) \in Y$ is frequently denoted by $\phi x$. Similarly, if  $\psi:Y\rightarrow Z$ is a function, the composite function 
$\psi \circ \phi:X \rightarrow Z$ is frequently denoted by $\psi \phi$.


\subsection{\bf Quasi-isometries and Almost-isometries}

\vskip 5pt
This subsection reviews basics concerning quasi-isometries.
\vskip 5pt
Let $(X,d_X)$ and $(Y,d_Y)$ be metric spaces.  Given constants $K\geq 1$ and $C\geq 0$, a map $\phi:X\rightarrow Y$ is $(K,C)$-\textit{quasi-isometric} if for every $x_1,x_2 \in X$,
\begin{equation}\label{quasiisometric}
(1/K)d_X(x_1,x_2)-C \leq d_Y(\phi x_1,\phi x_2)\leq Kd_X(x_1,x_2)+C.
\end{equation}
The map $\phi:X \rightarrow Y$ is $C$-\textit{coarsely onto} if for each $y \in Y$ there exists $x \in X$ with
\begin{equation}\label{coarseonto}
 d_Y(\phi x,y)\leq C.
 \end{equation}
The map $\phi :X \rightarrow Y$ is a $(K,C)$-\textit{quasi-isometry} when it is both $(K,C)$-quasi-isometric and $C$-coarsely onto.  The spaces $X$ and $Y$ are \textit{quasi-isometric} when such a quasi-isometry exists. Note that when $C=0$, quasi-isometries reduce to bi-Lipschitz maps.

A \textit{coarse inverse} to a $(K,C)$-quasi-isometric map $\phi :X \rightarrow Y$ is a $(K,C)$-quasi-isometric map 
$\psi :Y \rightarrow X$ satisfying 
\begin{equation}\label{inverse}
d_X\big((\psi \phi) x,x\big)\leq C\,\,\,\,\,\,\,\,\,\,\,\,\,\,\,\,\,\,\,\,\,\,\, d_Y\big((\phi \psi) y,y\big)\leq C
\end{equation}
for every $x\in X$ and for every $y\in Y$.  If $\phi :X \rightarrow Y$ is $(K,C)$-quasi-isometric with a coarse inverse 
$\psi :Y \rightarrow X$ then (\ref{inverse}) implies (\ref{coarseonto}) for both $\phi$ and $\psi$ so that they are 
both $(K,C)$-quasi-isometries.  

Conversely, given a $(K,C)$-quasi-isometry $\phi :X \rightarrow Y$, one uses (\ref{coarseonto}) to define a map $\psi :Y \rightarrow X$ satisfying
\begin{equation}\label{Gdef}
d_Y\big((\phi \psi) y,y\big)\leq C
\end{equation}
for each $y \in Y$.  The triangle inequality, (\ref{quasiisometric}), and (\ref{Gdef}) imply that for each $y_1,y_2 \in Y$ and $x \in X$,
\begin{equation}
(1/K)d_Y(y_1,y_2)-(3C/K) \leq d_X(\psi y_1,\psi y_2)\leq Kd_Y(y_1,y_2)+3KC,
\end{equation}
\begin{equation}
d_X\big((\psi \phi) x,x\big)\leq 2KC.
\end{equation}
Therefore $\psi :Y \rightarrow X$ is $(K,3KC)$-quasi-isometric, and letting $\bar{C}=3KC$, 
both $\phi$ and $\psi$ are $(K,\bar{C})$-quasi-isometric and coarse inverses of each other.

Define two maps $f, g:X \rightarrow Y$ to be equivalent if $\sup_{x\in X} d_Y(fx, gx \big)<\infty$, and denote this equivalence relation by $f\sim g$.  The discussion in the previous paragraph is summarized by: given any quasi-isometry $\phi : X\rightarrow Y$, there exists a quasi-isometry $\psi: Y\rightarrow X$ with the property that $\phi \psi \sim \Id_X$ and $\psi \phi \sim \Id_Y$.  Equivalence classes of self quasi-isometries of $X$ form a group, denoted by $QI(X)$. Quasi-isometries $\phi:X\rightarrow Y$  between spaces induce isomorphisms $QI(X)\cong QI(Y)$.

\vskip 10pt

In the special case where $K=1$, we replace the adjective quasi- with almost- throughout. In particular, the discussion 
above yields the following 

\begin{lem}\label{defB}
Let $\phi :X \rightarrow Y$ be a $(C/3)$-almost-isometry.  Then there exists a $C$-almost-isometry 
$\psi :Y \rightarrow X$ satisfying 
$d_X\big((\psi \phi) x,x\big)\leq C$ and $d_Y\big((\phi \psi) y,y\big)\leq C$ 
for every $x \in X$ and $y \in Y$.
\end{lem}

Since almost-isometries are special cases of quasi-isometries, the $\sim$ equivalence relation restricts to an equivalence
relation on almost-isometries. Compositions of almost-isometries are almost-isometries, and by Lemma 
\ref{defB}, coarse inverses of almost-isometries are almost-isometries. 
Therefore, equivalence classes of almost-isometries form a subgroup $AI(X)$ of $QI(X)$ with a canonical homomorphism $\textit{Isom}(X) \rightarrow AI(X)$.  Almost-isometries $\phi:X\rightarrow Y$ induce isomorphisms $AI(X) \cong AI(Y)$.

\begin{exmp}\label{example-EuclideanSpace}
If $X$ is a compact metric space, then any two maps have finite distance, so $QI(X)$ and $AI(X)$ are trivial. In contrast, $\textit{Isom}(X)$ can be quite non-trivial.  In particular, $\textit{Isom}(X) \rightarrow AI(X)$ need not be injective.

If $X=\mathbb R^n$ with the Euclidean metric, then $\textit{Isom}(X)=\mathbb{R}^n \rtimes O(n)$ where $O(n)=\{A \in GL(\mathbb{R}^n)\, \vert\, A^{T}A=\Id\}$ denotes the orthogonal group and $(v,A) \in \textit{Isom}(X)$ acts via $w \mapsto Aw+v$.  The natural homomorphism $\textit{Isom}(X) \rightarrow AI(X)$ has kernel given by the translations $\mathbb{R}^n$ and image isomorphic to $O(n)$.
\end{exmp}


\subsection{\bf (Quasi)-Isometries of CAT(-1) spaces.}

\vskip 5pt

This subsection summarizes the basic theory of isometries, quasi-isometries, and boundary maps of CAT(-1) spaces; the reader is referred to \cite[Chapter II.6]{BH} for more details. 

\vskip 5pt

Throughout $X$ denotes a CAT(-1) metric space that is {\it proper}: {\it all metric balls are 
compact}. A group $G$ acting on $X$ acts \textit{geometrically} provided the $G$-action on $X$ is isometric, proper, free, and cocompact. 

\subsubsection{ (Bounded) Isometries.\\}
For $I\in \textit{Isom}(X)$, the displacement function $d_{I}:X\rightarrow \mathbb{R}$ is defined by $d_I(x)=d(Ix,x)$.  The isometry $I$ is defined to be a \textit{bounded isometry} if $d_I$ is a bounded function.  The translation length of $I$, denoted by $\tau(I)$, is defined by $\tau(I)=\inf_{x \in X} d_I(x)$. The set of points where $d_I$ achieves its infimum is denoted $\Min(I)$. An isometry $I$ is 
{\it semi-simple} if $\Min(I) \neq \emptyset$. If $G$ acts geometrically on $X$, then every $g \in G$ acts via a semi-simple isometry. For a semi-simple isometry $I$,
\begin{equation}\label{formula}
\tau(I)=\lim_{n \rightarrow \infty} \frac{d(x,I^nx)}{n}
\end{equation}
where $x\in X$ is an arbitrary point \cite[II.6, Exercise 6.6(1)]{BH}. 

\begin{lem}\label{nobounded}
Let $X$ be a $CAT(-1)$ space, not isometric to $\mathbb R$. Then $X$ has no nontrivial bounded isometries. \end{lem}

\begin{proof}
Assume that $I \in \Isom(X)$ is bounded.  Then the displacement function $d_I$ is bounded and convex \cite[II.6, Proposition 6.2(3)]{BH}, hence constant. In particular, $\Min(I)=X$.

If this constant is positive, then $\Min(I)$ splits isometrically as a metric product $Y\times \mathbb R$, for some convex subset $Y\subset X$ \cite[II.6, Theorem 6.8(4)]{BH}. $Y$ cannot consist of a single point (since $X$ is not isometric to $\mathbb R$), nor can it have more than one point (for otherwise, $X$ contains an isometric copy of $[0, \epsilon]\times \mathbb R$, so is not CAT(-1)). Conclude that $d_I \equiv 0$ and that $I=\Id_X.$ 
\end{proof}

\begin{cor}\label{bdd-distance-coincide}
Let $X$ be a $CAT(-1)$ space, not isometric to $\mathbb R$, and let $g,h\in \textit{Isom}(X)$. 
If $\sup_{x\in X}d\big(gx, hx\big)<\infty$, then $g=h$. In particular,
for such spaces, the natural map $\textit{Isom}(X) \rightarrow AI(X)$ is injective.

\end{cor}

\begin{proof}
Apply Lemma \ref{nobounded} to $gh^{-1}$.
\end{proof}

\subsubsection{ Boundary structure of $X$.\\}

\vskip 5pt

The boundary $\partial X$ of $X$ consists of the set of equivalence classes of geodesic
rays in $X$, where two rays are equivalent if they are at bounded Hausdorff distance. There is a natural topology on $\partial X$, where two geodesic rays based at $x_0 \in X$ are close provided they stay close for a long period of time. A quasi-isometry $\phi: X\rightarrow Y$ induces a homeomorphism $\phi^\partial: \partial X \rightarrow \partial Y$.

A pair of maps $f, g:X \rightarrow X$ are at distance at most 
$L$ when $d(fx,gx)\leq L$ for every $x \in X$. The boundary at infinity detects whether maps are
at bounded distance apart. More precisely, we have the following well-known result:

\begin{prop}\label{L}
Let $(X,d)$ be a complete simply connected $CAT(-1)$ metric space, having the geodesic extension property, 
and with the property that the space of directions at each point is connected.  
For each $K\geq 1$ and $C>0$ there exists a constant $L:=L(K,C)>0$ with the following property: if $F$ is a 
$(K,C)$-quasi-isometry of $X$ and $I$ is an isometry of $X$ with boundary maps $ F^\partial\equiv I^\partial$, 
then $F$ and $I$ are at distance at most $L$.
\end{prop}

We were unable to locate a proof in the literature, other than in the special case where $(X,d)$ is a negatively 
curved Riemannian manifold (which was shown by  Pansu \cite[Lemma 9.11, pg. 39]{Pa}). For the convenience
of the reader, we provide a proof which closely follows Pansu's Riemannian argument.

\begin{proof}
Without loss of generality, we may assume that $I=\text{Id}_X$. 
Let $p\in X$ be an arbitrary point, and consider the geodesic segment $\eta$ from $p$ to $F(p)$, whose length
we would like to uniformly control. The segment defines
a point $x_-$ in the space of directions $S_p$ at the point $p$. From the geodesic extension property, we can
extend this geodesic beyond $p$, which defines a second point $x_+$ on $S_p$. In terms of the Alexandrov 
angular metric 
$\angle_p$ on the space of directions $S_p$ (see \cite[Definition II.3.18, pg. 190]{BH}), we have that 
$\angle_p(x_+, x_-) = \pi$ (as they correspond to a geodesic through $p$). 
Since the space of directions $S_p$ is connected, continuity now implies the existence of a point $y_-\in S_x$ with the
property that $\angle_p(x_+, y_-) = \pi/2 = \angle_p(x_-, y_-)$. Let $\gamma$ be a geodesic segment terminating on $p$, and 
representing $y_-$. By the geodesic extension property, we can extend $\gamma$ to a bi-infinite geodesic $\hat \gamma$. 
The continuation of $\gamma$ defines a second point $y_+\in S_p$; again, we have $\angle_p(y_+, y_-)=\pi$. 
We now claim that the point $p$ coincides with the projection point of $F(p)$ on the geodesic $\hat \gamma$.

To see this, recall that in a CAT(-1) space, there is uniqueness of the projection point $\rho(q)$ of a point $q$ onto
a closed convex subset $C$. Moreover, the point $\rho(q)$ is characterized by the following property:
the angle at $\rho(q)$ between the geodesic segment from $\rho(q)$ to $q$ and any other 
geodesic segment originating at $\rho(q)$ {\it in the set $C$} is at least $\pi/2$ (see e.g. \cite[Proposition II.2.4, pg. 176]{BH}). 
We apply this criterion to the 
convex set $\hat \gamma$, and the point $F(p)$. Locally near $p$ there are precisely two geodesics segments
in $\hat \gamma$, corresponding to the pair of directions $y_+, y_- \in S_p$. We already know that 
$\angle_p(y_-, x_-)=\pi/2$, so it suffices to verify that $\angle_p(y_+, x_-)\geq \pi/2$. But this is clear, for 
otherwise the triangle inequality would force a contradiction:
$$\pi=\angle_p(y_+, y_-) \leq \angle_p(y_+, x_-) + \angle_p(x_-, y_-)< \pi/2 + \pi/2.$$
So $p$ is indeed the closest point to $F(p)$ on the geodesic $\hat \gamma$.

Now apply the map $F$ to obtain the $(K, C)$-quasi-geodesic $F\circ \hat \gamma$. From the stability theorem for
quasi-geodesics (see \cite[Theorem III.H.1.7, pg. 401]{BH}) 
there is a uniform constant $L:=L(K, C)$, depending only on the constants $K,C$ for the quasi-geodesic, with the property 
that $F\circ \hat \gamma$ is at Hausdorff distance $\leq L$ from the geodesic with same endpoints on $\partial X$, which is
$\hat \gamma$. It follows that the point $F(p) \in F\circ \hat \gamma$ is at distance $\leq L$ from $\hat \gamma$. But from
the discussion in the previous paragraph, this implies $d(p, F(p))\leq L$, as desired.
\end{proof}

In fact, there is some additional {\it metric} structure on $\partial X$: fixing a basepoint $x\in X$, define the visual metric
$$d_{\partial X}(p, q) =e^{- (p|q)_x}$$
where $p, q\in \partial X$, and $(p|q)_x$ denotes the Gromov product of the pair of points with respect to the basepoint
$x$ (see \cite[Section 2.5]{Bou} for details). While the metric $d_{\partial X}$ 
depends on the choice of basepoint $x$, changing basepoints gives a bi-Lipschitz equivalent 
metrics, and hence
the {\it bi-Lipschitz class} of the metric $d_{\partial X}$ is well-defined. 

Fixing such metrics on $\partial X, \partial Y$,
the behavior of a quasi-isometry $\phi: X\rightarrow Y$ is closely related to the metric properties of the induced map
$\phi ^\partial : \partial X \rightarrow \partial Y$. Most relevant for our purposes is work of Bonk and Schramm, who 
showed that if $\phi$ is an almost-isometry, the $\phi^\partial$ is a bi-Lipschitz map \cite[proof of Theorem 6.5]{BS}, 
i.e. there is a constant $\lambda >1$ with
the property that for all $x,y\in \partial X$ we have:
$$\lambda ^{-1} \cdot d_{\partial X}(x, y)\leq d_{\partial Y}\left( \phi^\partial(x), \phi^\partial(y) \right)
\leq \lambda \cdot d_{\partial X}(x,y)$$
Conversely, if $\phi^\partial$ is a bi-Lipschitz map, then $\phi$ is at bounded distance from an almost-isometry \cite[Theorems 7.4 and 8.2]{BS}. In particular, boundary maps induce an isomorphism $AI(X) \cong Bilip(\partial X)$.

\subsection{\bf Rigidity statements.}
In the statement of our Main Theorem, our hypotheses involve some rigidity statements concerning the spaces $X$, $Y$. We define these rigidity statements in this subsection for the convenience of the reader.

\begin{defn}
A metric space $Y$ is \emph{quasi-isometrically rigid} (QI-rigid) if each quasi-isometry of $Y$ is at a 
bounded distance from an isometry of $Y$. In other words, the canonial homomorphism $ \textit{Isom} (Y) \rightarrow QI(Y)$ 
is surjective. A metric space is \emph{almost-isometrically rigid} (AI-rigid) if every almost-isometry of $Y$ is at bounded 
distance from an isometry. In other words, the canonical homomorphism $\textit{Isom} (Y)\rightarrow AI(Y)$ is surjective.
\end{defn}

A celebrated result of Pansu \cite{Pa} shows that quaternionic hyperbolic space $\mathbb O \mathbb H ^n$ 
(of real dimension $4n$) and the Cayley hyperbolic plane $\text{Ca}\mathbb H^2$ (of real dimension $16$) are both 
QI-rigid (and hence AI-rigid). In contrast, we have the following:

\begin{lem}
For any $n\geq 2$, $\mathbb H^n$ is {\bf not} AI-rigid. In other words, there exist almost-isometries 
$\phi: \mathbb H^n\rightarrow \mathbb H^n$ which are {\bf not} at bounded distance from any isometry.
\end{lem}

\begin{proof}
From the discussion in the previous section, one can think of this entirely at the level of the metric structure on the
boundary at infinity. Choosing the disk model for $\mathbb H^n$ and the basepoint $x$ to be the origin, the metric 
$d_{\partial \mathbb H^n}$ on
$\partial \mathbb H^n = S^{n-1}$ is conformal to the standard (round) metric on the sphere -- in fact, 
$d_{\partial \mathbb H^n}(p, q)$ is half the (Euclidean) length of the (Euclidean) segment joining $p$ to $q$
(see \cite[Example 2.5.9]{Bou}). Recalling that almost-isometries induce bi-Lipschitz maps \cite[Theorem 6.5]{BS}, 
while isometries induce conformal maps \cite[Corollaire 2.6.3]{Bou}, the lemma follows immediately from the fact
that there exist bi-Lipschitz maps $\phi^\partial: S^{n-1}\rightarrow S^{n-1}$ which are not conformal.
\end{proof}

Similarly, one can show that complex hyperbolic space $\mathbb C \mathbb H ^n$ (of real dimension $2n$) 
is {\bf not} AI-rigid. Let us mention a few further examples.

\begin{exmp}\label{AI-rigid-line}
Consider $\mathbb R$ with the standard metric. From the discussion in Example \ref{example-EuclideanSpace}, we
have that the image of $\textit{Isom} (\mathbb R)$ inside $QI(\mathbb R)$ is a copy of $\mathbb Z_2$ (with non-trivial element
represented by the map $\sigma$ defined via $\sigma(x)=-x$). For any $\lambda >0$, 
the map $\mu_\lambda : x\mapsto \lambda x$ is a quasi-isometry, and if $\lambda \neq \lambda^{'}$, then $\mu _\lambda
\not \sim \mu_{\lambda ^{'}}$. So $QI(\mathbb R)$ at the very least contains the continuum many distinct elements 
$[\mu_\lambda]$, and the map $\textit{Isom} (\mathbb R) \rightarrow QI(\mathbb R)$ is far from being surjective.

On the other hand, assume $\phi :\mathbb R \rightarrow \mathbb R$ is a $C$-almost-isometry. Up to composing with 
$\sigma$, we may assume that $\phi$ preserves the two ends of $\mathbb R$, and up to composing with a translation, 
we may assume $\phi(0)=0$. Let us estimate the distance from $\phi x$ to $x$ for a generic $x\in \mathbb R$.
Firstly, if $x>0$ is sufficiently large, we have that $\phi x >0$ (since
$\phi$ preserves the ends of $\mathbb R$) and since $\phi$ is a $C$-almost-isometry,
$|\phi x - x| = ||\phi x - \phi 0| - |x -0| | \leq C$. An identical argument shows that if $x<0$ is sufficiently negative,
then $|\phi x - x| \leq C$. This leaves an $R$-neighborhood $B$ of the fixed point 
$0$ (for some $R$). But for $x\in B$, we know that $\phi x$ has distance at most $R+C$ from the origin, so the
triangle inequality gives $|x - \phi x| \leq 2R+C$. It follows that $\sup _{x\in \mathbb R} ( \phi x, x) \leq 2R + C$, 
and hence $\phi \sim \Id _{\mathbb R}$. This shows that every almost-isometry of $\mathbb R$ lies at finite distance 
from an isometry. Hence $\mathbb R$ is an example of an AI-rigid space which is {\it not} 
QI-rigid. 
\end{exmp}

\begin{exmp}\label{AI-rigid-Euclidean-plane}
As a somewhat more sophisticated example, consider now the case of $\mathbb R^2$ with a flat metric. We claim that 
$\mathbb R^2$ is an AI-rigid space, i.e. that every self almost-isometry is at bounded distance from an isometry. To see 
this, we start with $F\in AI(\mathbb R^2)$ arbitrary, and try to find a {\it standard form} almost-isometry at bounded distance
from $F$. Note that by composing with a translation, we may assume $F(0)=0$, and at 
the cost of a bounded perturbation, we can also assume that $F$ is continuous. We will find it convenient to work
in polar coordinates $(r, \theta)$. 

Since $F(0)=0$, we see that $F$ maps the circle $r = R$ into the annular region $R-C \leq r \leq R+C$. Performing
a radial projection of the image onto the circle of radius $R$ results in a new map at bounded distance from $F$ (hence
a new almost-isometry), which has the additional property that $F$ maps each circle about the origin to itself. So
without loss of generality, we may assume that $F$ has the form $F (r, \theta) = \left(r, f(r, \theta)\right)$ for some continuous 
function $f$; let $\alpha: \mathbb R^+\rightarrow \mathbb R$ denote the function $\alpha (r) := f(r, 0)$. 
Now consider the points $(r, 0)$ on the ray $\theta =0$, and observe that each of these gets sent to a point 
$\left(r, \alpha (r) \right)$. On the circle $S(R)$ of radius $r = R$ centered at the origin, the map $\phi$ is at bounded 
distance from the rotation by an angle $\alpha(R)$ -- moreover, the distance between the two maps is bounded 
independently of the radius $R$. It follows that the map $F$ is at bounded distance from the map $(r, \theta) \mapsto 
\left(r, \theta + \alpha(r)\right)$. 

Next, let us focus on properties of the map $\alpha$. The ray $\theta = 0$ maps under the almost-isometry $F$ to the
path $(r, \alpha (r))$. We now estimate the angle $\rho(s, t)$ ($s<t$) from the origin between the points
$(s, \alpha(s))$ and $(t, \alpha(t))$ 
-- which is obviously $\alpha(t) - \alpha(s)$ -- via the law of cosines:
$$\cos\big(\rho(s,t) \big):= \frac{s^2 + t^2 - ||(t, \alpha(t)) - (s, \alpha(s))||^2}{2st}$$
But since the map $F$ is a $K$-almost isometry, we have the estimate
$$t-s - K \leq ||(t, \alpha(t)) - (s, \alpha(s))||\leq t-s + K$$
which upon substitution gives the estimate
$$1- \frac{K^2 + 2K(t-s)}{2st} \leq \cos\big(\rho(s,t) \big) \leq 1- \frac{K^2 - 2K(t-s)}{2st}.$$
These bounds tend to $1$ as $s < t$ both tend to infinity. Moreover, for any $\epsilon >0$, we can find an 
$s_0$ with the property that for any $t>s_0$, the lower bound is at least $1-\epsilon$. 
This implies that $\alpha(r)$ has a limit. Let $\alpha_\infty$ denote the limit $\lim _{r\to \infty}\alpha(r)$, and observe that, at the cost of composing
with a rotation by $-\alpha_\infty$, we may as well assume that $\lim _{r \to \infty} \alpha(r) = 0$.
So we have reduced the problem to the following special case: let $F:(r, \theta) \mapsto \left(r, \theta + \alpha(r)\right)$ be
a $K$-almost-isometry, where $\alpha: \mathbb R^+ \rightarrow \mathbb R$ is a continuous map with 
$\lim _{r \to \infty} \alpha(r) = 0$. We need to show that this map $F$ is at bounded distance from the identity map -- it is sufficient to prove that, for $r$ sufficiently large, $\alpha(r) \leq K^\prime / r$ (for some constant $K^\prime$).

Consider the pair of points $(r_1, 0)$ and $(r_2, \theta)$ on the plane, and their image under the 
$K$-almost-isometry. The distance between the two pairs of points is easily calculated from the law of cosines, and
the $K$-almost-isometry condition gives the following estimate
$$\Big|\sqrt{r_1^2 + r_2^2 -2r_1r_2\cos(\theta + \Delta \alpha(r_1, r_2))} - \sqrt{r_1^2 + r_2^2 - 2r_1r_2\cos(\theta)} \Big|\leq K$$
which can be rewritten as
$$\Big|\sqrt{1 -\frac{2r_1r_2}{r_1^2 + r_2^2}\cos(\theta + \Delta \alpha(r_1, r_2))} - \sqrt{1 - \frac{2r_1r_2}{r_1^2 + r_2^2} \cos(\theta) } \Big|\leq \frac{K}{\sqrt{r_1^2 + r_2^2}}$$
where $\Delta\alpha$ is the difference function associated to $\alpha$, i.e. $\Delta\alpha(s,t) = \alpha(t)-\alpha(s)$. Now fix
a $0< \lambda <1$, and specialize the above equation to the case where $r_2 = r$ and $r_1=\lambda r$ ($r$ will be taken to
tend to infinity), and $\theta$ is a fixed
constant chosen so that $\sin(\theta) \neq 0$. We obtain:
$$\Big|\sqrt{1 -\frac{2\lambda}{1+\lambda^2}\cos(\theta + \phi(r))} - \sqrt{1 - \frac{2\lambda}{1+\lambda^2} \cos(\theta)} \Big|\leq \frac{K}{r\sqrt{1+\lambda ^2}}$$
where $\phi(r):= \Delta \alpha(\lambda r, r) = \alpha(r) -\alpha(\lambda r)$ tends to $0$ as $r\to \infty$. Using the sum-angle
formula for cosine, and a Taylor approximation for the terms involving $\phi(r)$, the left hand side can be rewritten as
$$\Big|\sqrt{1 -\frac{2\lambda}{1+\lambda^2}\cos(\theta) + \phi(r) \frac{2\lambda\sin(\theta)}{1+\lambda^2}
+ o(\phi(r))} - \sqrt{1 - \frac{2\lambda}{1+\lambda^2} \cos(\theta)} \Big|$$
Recalling that $\lambda, \theta$ are fixed, while $\phi(r) \to 0$ as $r\to \infty$, we can use a Taylor expansion for the
function $g(x) = \sqrt{a + x} \approx \sqrt{a} + x/2\sqrt{a} + o(x)$. Substituting in, the left hand side further reduces, and we
obtain
$$\Big| \phi(r) \left(\frac{\lambda \sin(\theta)}{(1+\lambda^2)\sqrt{1-2\lambda\cos(\theta)(1+\lambda^2)^{-1}}}\right) +o(\phi(r))\Big| \leq \frac{K}{r\sqrt{1+\lambda ^2}}$$
which gives us the asymptotic estimate $|\phi(r)| \leq K^{\prime \prime}/r$ (for $r$ sufficiently large), 
where $K^{\prime \prime}$ is a constant satisfying 
$$K^{\prime \prime} > K\frac{\sqrt{1+\lambda^2 -2\lambda\cos(\theta)}}{\lambda \sin(\theta)}.$$
Finally, recalling that $\phi(r) := \alpha(r) - \alpha(\lambda r)$, that $\lim _{s\to \infty} \alpha(s) =0$, and that $0< \lambda <1$,
we can use a telescoping sum to obtain the estimate:
$$|\alpha(r)| = \lim _{s\to \infty}|\alpha(s) - \alpha(r)|  \leq   \sum_{i=0}^\infty |\alpha(\lambda ^{-i-1} r) - \alpha(\lambda ^{-i} r)| \leq \sum_{i=0}^\infty 
\frac{K^{\prime \prime}}{\lambda^{-i} r} = \frac {K^{\prime \prime}}{r(1-\lambda)}.
$$
Since $K^{\prime \prime}, \lambda$ are fixed constants, this gives the desired asymptotic estimate on the rotation 
function $\alpha(r)$, completing the argument.
\end{exmp}

\begin{exmp}\label{perturb-hyperbolic-plane}
Consider $\mathbb H^2$ with the standard hyperbolic metric of constant curvature $-1$. Taking a compact set 
$K\subset \mathbb H^2$, perturb the metric slightly in the compact set $K$, and call the resulting Riemannian manifold $X$.
If the perturbation is small enough, $X$ will be negatively curved, and one can arrange for $\textit{Isom}(X)$ to
be trivial. 

Let $\phi: \mathbb H^2 \rightarrow X$ be the identity map, and note that $\phi$ is an almost-isometry from $\mathbb H^2
\rightarrow X$ (though there are no isometries from $\mathbb H^2$ to $X$). It follows that $AI(X) \cong AI(\mathbb H^2)$, and
we know from Lemma \ref{nobounded} that the map
$\textit{Isom}(\mathbb H^2)\hookrightarrow AI(\mathbb H^2)$ is injective. Hence the group $AI(X)$ contains a copy of 
$PSL(2, \mathbb R)$, and the map $\textit{Isom}(X)\rightarrow AI(X)$ fails to be surjective.
\end{exmp}

\begin{defn}
A complete CAT(-1) space $X$ equipped with a geometric $G$-action $\rho: G\rightarrow 
\textit{Isom}(X)$ is {\it marked length 
spectrum rigid} (MLS-rigid) provided: anytime we are given a complete CAT(-1) space $Y$ equipped with a geometric 
$G$-action $i: G\rightarrow \textit{Isom}(Y)$, and the translation lengths satisfy $\tau(\rho(g))=\tau(i(g))$ for each $g \in G$, then
there exists a $(\rho, i)$-equivariant isometric embedding $X\hookrightarrow Y$.
\end{defn}

\begin{rem}
When considering the MLS-rigidity question, one can also formulate versions where, rather than allowing an arbitrary
CAT(-1) space $Y$, one restricts to a certain subclass $\mathcal F$ of CAT(-1) spaces. In this case, we say that $X$ is 
MLS-rigid within the class $\mathcal F$. For instance, if $X$ is a 
negatively curved Riemannian manifold, it is reasonable to focus on the case where $Y$ is also a negatively curved 
Riemannian manifold. In this case, the conclusion forces the embedding to be surjective, and hence the equivariant
embedding is automatically an isometry from $X$ to $Y$. This is the context of the classical MLS Conjecture.
\end{rem}






\section{\bf Proof of Theorem \ref{theorem-AI+MLS} and applications}\label{section-proof-of-main}


Throughout this section, we assume that $X$ and $Y$ satisfy the hypotheses of Theorem \ref{theorem-AI+MLS}. 
Let us briefly sketch
out the main steps of the proof. First, we use the almost-isometry between $X$ and $Y$ to transfer the isometric 
$G$-action on $X$ to an almost-isometric $G$-action on $Y$. Using the property that $Y$ is almost-isometrically
rigid, one can straighten the almost-isometric $G$-action on $Y$ to a {\it genuine} isometric $G$-action on $Y$. We
then verify that this new isometric action on $Y$ is also geometric. Such a construction of a geometric $G$-action
on $Y$ is likely well-known -- we include the details for the convenience of the reader.
Now with respect to this new action on $Y$, one can construct an {\it equivariant} almost-isometry between $X$ and $Y$.
It is easy to check that these two actions have the same translation lengths, so from the marked length rigidity of $X$
we obtain the isometric embedding $X \hookrightarrow Y$. We now give the details of the proof.





\subsection{Pushing forward the action}

As $X$ and $Y$ are almost-isometric, there exists a $(C/3)$-almost-isometry $\phi :X \rightarrow Y$.  In particular 
$\phi$ is a $C$-almost-isometry.  By Lemma \ref{defB} there exists a $C$-almost-isometry coarse inverse 
$\psi :Y \rightarrow X$ satisfying   
$$d_X\big((\psi \phi) x,x\big)\leq C \,\,\,\,\,\,\,\,\,\,\,\,\,\,\,\,\,\,\,\,\,\,\, d_Y\big((\phi \psi) y,y\big)\leq C$$ 
for every $x \in X$ and $y \in Y$.

\vskip 10pt





Recall that $G<\Isom(X)$ acts properly discontinuously, freely, and cocompactly on $X$.  For $g\in G$ define the
map $\bar{g}:Y\rightarrow Y$ by $\bar{g}=\phi g \psi$. In other words, $\bar{g}$ is chosen
to make the following diagram commute:
$$
\xymatrix{
Y \ar[r]^{\psi} \ar[d]_{\bar{g}} & X \ar[d]^{g}\\
Y & X \ar[l]^{\phi} \\
}
$$

\begin{lem}\label{barg}
For each $g \in G$, $\bar{g}$ is a $3C$-almost-isometry of $Y$.
\end{lem}

\begin{proof}
Let $y_1,y_2 \in Y$.  We verify
\begin{eqnarray} d_Y(\bar{g}y_1,\bar{g}y_2)&=&d_Y\big((\phi g \psi) y_1 , (\phi g \psi) y_2\big)\nonumber\\
&\leq& d_X(g \psi y_1 , g \psi y_2)+C\nonumber\\
&=&d_X(\psi y_1 , \psi y_2)+C\nonumber\\
&\leq& d_Y(y_1,y_2)+2C\nonumber.  
\end{eqnarray} 


A symmetric argument shows that $d_Y(\bar{g}y_1,\bar{g}y_2) \geq d_Y(y_1,y_2)-2C$, giving us that
$$d_Y(y_1,y_2)-3C\leq d_Y(\bar{g}(y_1),\bar{g}(y_2))\leq d_Y(y_1,y_2)+3C.$$  

\noindent It remains to show that $\bar{g}$ is $3C$-coarsely onto.  For $y \in Y$, let $y'=\overline{g^{-1}} y = \phi g^{-1} \psi y$.  Then 
\begin{eqnarray}
d_Y(\bar{g}y',y)&=&d_Y\big((\phi g \psi) (\phi g^{-1} \psi y) , y\big)\nonumber\\
&\leq&d_X(\psi \phi g \psi \phi g^{-1} \psi y , \psi y)+C\nonumber\\
&\leq&d_X\big((\psi \phi)(g \psi \phi g^{-1} \psi y) , g \psi \phi g^{-1} \psi y\big) + d_X(g\psi \phi g^{-1} \psi y , \psi y)+C\nonumber\\
&\leq&d_X(g\psi \phi g^{-1} \psi y , \psi y)+2C\nonumber\\
&=&d_X\big((\psi \phi)(g^{-1} \psi y) , g^{-1} \psi y)+2C\nonumber\\
&\leq&3C.\nonumber
\end{eqnarray}
The first inequality comes from $\psi$ being a $C$-almost-isometry, the second is the triangle inequality, 
the third and fourth both come from $\psi \phi \sim \Id_X$. This completes the proof of the Lemma.
\end{proof}

As $Y$ is AI-rigid, Lemma \ref{barg} and Proposition \ref{L} yield a constant $L>0$ such that for 
each $g \in G$, there is a unique isometry $i(g) \in \Isom(Y)$ satisfying 
$$d_Y\big(\bar{g}y , i(g)y \big)\leq L$$ 
for every $y \in Y$. It is important to note that the constant $L$ is {\it independent} of the choice of element $g$
(this is used in the proof of Lemma \ref{propdisc}). This defines a map $i: G\rightarrow \Isom(Y)$. 

\begin{lem}\label{homomorphism}
The map $i:G \rightarrow \Isom(Y)$ is a homomorphism.
\end{lem}

\begin{proof}
Let $g_1,g_2 \in G$; we want to compare the elements $i(g_1 g_2)$ and $i(g_1) i(g_2)$ inside $\Isom (Y)$.  
By Corollary \ref{bdd-distance-coincide} (since $Y\neq \mathbb R$, as the space of directions of $Y$ is connected), it suffices to show that $i(g_1 g_2) \sim i(g_1) i(g_2)$. 
So we need to estimate the effect of these two isometries on a generic element $y\in Y$. As a preliminary estimate we have
\begin{eqnarray}
d_Y(\overline{g_1g_2}y,\bar{g_1}\bar{g_2}y)&=&d_Y\big((\phi g_1g_2 \psi)y , (\phi g_1 \psi)(\phi g_2\psi)y\big)\nonumber\\
&\leq&d_X\big(g_1(g_2 \psi y) , g_1(\psi \phi g_2 \psi y)\big)+C\nonumber\\
&=&d_X\big(g_2\psi y , (\psi \phi) (g_2 \psi y)\big)+C\nonumber\\
&\leq&2C.\nonumber
\end{eqnarray}
The first inequality uses that $\phi$ is a $C$-almost-isometry, while the second uses that $\psi \phi \sim \Id_X$.
Using this, we can now estimate:
\begin{eqnarray}
&& \hskip -40pt d_Y\big(i(g_1g_2)y,i(g_1)i(g_2)y\big)\nonumber\\
&\leq& d_Y\big(i(g_1g_2)y,(\overline{g_1g_2})y\big)+d_Y\big(\overline{g_1g_2}y,i(g_1)i(g_2)y\big)\nonumber\\
&\leq& d_Y\big(\overline{g_1g_2}y,i(g_1)i(g_2)y\big)+L\nonumber\\
&\leq& d_Y\big(\overline{g_1g_2}y,\bar{g_1}\bar{g_2}y\big)+d_Y\big(\bar{g_1}\bar{g_2}y,i(g_1)i(g_2)y\big)+L\nonumber\\
&\leq& d_Y\big(\bar{g_1}\bar{g_2}y,i(g_1)i(g_2)y\big)+2C+L\nonumber\\
&\leq& d_Y\big(\bar{g_1}(\bar{g_2}y),i(g_1)(\bar{g_2}y)\big)+d_Y\big(i(g_1)\bar{g_2}y,i(g_1)i(g_2)y\big)+2C+L\nonumber\\
&\leq& d_Y\big(i(g_1)\bar{g_2}y,i(g_1)i(g_2)y\big)+2C+2L\nonumber\\
&=&d_Y\big(\bar{g_2}y,i(g_2)y\big)+2C+2L\nonumber\\
&\leq&2C+3L.\nonumber
\end{eqnarray}
Since this estimate holds for arbitrary $y\in Y$, we conclude $i(g_1g_2)\sim i(g_1)i(g_2)$.
Applying Corollary \ref{bdd-distance-coincide}, this gives us $i(g_1g_2)=i(g_1)i(g_2)$, establishing the Lemma.

\end{proof}


\subsection{Verifying the new action is geometric} Now that we have constructed a homomorphism $i: G\rightarrow 
\text{Isom}(Y)$, our next step is to show that this $G$-action on $Y$ is geometric.

\begin{lem}\label{inj}
The homomorphism $i:G \rightarrow \Isom(Y)$ is injective.
\end{lem}

\begin{proof}
Let $g \in G$ and assume that $i(g)=\Id_Y$.  By Lemma \ref{nobounded}, it suffices to show that $g\sim \Id _X$, 
so we need to estimate how far $g$ moves a generic element $x\in X$. First, observe that for each $y \in Y$, 
$$d_Y(\bar{g}y,y)=d_Y(\bar{g}y,i(g)y)\leq L.$$  
We can now estimate how far $g$ moves elements of the form $\psi y$: 
\begin{eqnarray}
d_X\big(g(\psi y) , \psi y\big)&\leq& d_Y\big(\phi (g \psi y) , \phi (\psi y)\big)+C\nonumber\\
&=&d_Y(\bar{g}y , \phi \psi y)+C\nonumber\\
&\leq&d_Y(\bar{g}y,y)+d_Y\big(y , (\phi \psi) y\big)+C\nonumber\\
&\leq&L+2C.\nonumber
\end{eqnarray}
Now for a generic $x \in X$, we have that $\phi x \in Y$, so we can estimate:
\begin{eqnarray}
d_X(gx,x)& \leq& d_X\big(gx,g(\psi \phi)x\big)+d_X\big(g(\psi \phi)x,x\big)\nonumber\\
&=&d_X\big(x,(\psi \phi)x\big)+d_X\big(g(\psi \phi)x,x\big)\nonumber\\
&\leq&d_X\big(g(\psi \phi)x,x\big)+C\nonumber\\
&\leq&d_X\big(g(\psi \phi)x, (\psi \phi)x)+d_X\big( (\psi \phi) x,x\big)+C\nonumber\\
&\leq&L+2C+C+C.\nonumber
\end{eqnarray}
This shows $g\sim \Id _X$, so by Lemma \ref{nobounded} (and using the hypothesis that $X\neq \mathbb R$) 
we conclude $g = \Id_X$, as claimed.
\end{proof}

\begin{lem}\label{propdisc}
The $G$-action on $Y$ is proper.
\end{lem}

\begin{proof}
We argue by contradiction.  If not, then there exist a $d_Y$-metric ball $B_Y \subset Y$ and an infinite sequence of distinct 
elements $i(g_j) \in i(G)$ with $i(g_j)B_Y \cap B_Y \neq \emptyset$ for each index $j$.  For each index $j$, choose 
$y_j,k_j \in B_Y$ such that $i(g_j)y_j=k_j$.  Let $D=\diam(B_Y)$.

Consider the closed $d_X$-metric ball  $B_X=\{x \in X\,\vert \, d_X(x,\psi y_0)\leq L+2D+5C\}$.  As the $G$ action on 
$X$ is proper, we obtain a contradiction by showing $g_j(\psi y_j) \in g_jB_X\cap B_X$ for each index $j$.
First note that for each $j$, 
$$d_Y(i(g_j)y_j,y_j)=d_Y(k_j,y_j)\leq D.$$  
Next we estimate how far each $g_j$ displaces the corresponding $\psi y_j$:
\begin{eqnarray}
d_X\big(g_j(\psi y_j),\psi y_j\big)&\leq&d_Y\big(\phi (g_j \psi y_j) , \phi (\psi y_j)\big)+C\nonumber\\
&=&d_Y\big(\bar{g_j}y_j, \phi \psi y_j\big)+C\nonumber\\
&\leq&d_Y(\bar{g_j}y_j,y_j)+d_Y\big(y_j, (\phi \psi) y_j\big)+C\nonumber\\
&\leq&d_Y(\bar{g_j}y_j,y_j)+2C\nonumber\\
&\leq&d_Y\big(\bar{g_j}y_j,i(g_j)y_j\big)+d_Y\big(i(g_j)y_j,y_j\big)+2C\nonumber\\
&\leq& L+D+2C.\nonumber
\end{eqnarray} 
We can now show  $\psi y_j \in B_X$ for each $j$, since
\begin{eqnarray}
d_X(\psi y_0,\psi y_j)&\leq&d_Y\big(\phi (\psi y_0),\phi (\psi y_j)\big)+C\nonumber\\
&\leq&d_Y\big((\phi \psi) y_0,y_0\big)+d_Y\big(y_0,\phi \psi y_j\big)+C\nonumber\\
&\leq&d_Y\big(y_0,\phi \psi y_j\big)+2C\nonumber\\
&\leq&d_Y\big(y_0,y_j\big)+d_Y\big(y_j,(\phi \psi) y_j\big)+2C\nonumber\\
&\leq&D+3C.\nonumber
\end{eqnarray}
Combining the above two inequalities, we obtain that $g_j(\psi y_j) \in B_X$, since 
$$d_X\big(g_j(\psi y_j),\psi y_0\big)\leq d_X\big(g_j(\psi y_j), \psi y_j)\big)+d_X\big( \psi y_j, \psi y_0)\leq (L+D+2C)+(D+3C).$$ 
We conclude that $g_j(\psi y_j) \in g_jB_X \cap B_X$ for each index $j$ as claimed above. This yields the desired contradiction,
and completes the proof of the Lemma.
\end{proof}

\begin{cor}\label{free}
The $G$-action on $Y$ is free.
\end{cor}

\begin{proof}
If not, then there exists a nonidentity element $g \in G$ and a point $y \in Y$ with $i(g)y=y$.  Note that 
$G$ is torsion-free (as the $G$-action on $X$ is free), and $i$ is 
injective by Lemma \ref{inj}, so $i(g)\in \Isom(Y)$ also has infinite order. This gives infinitely many elements fixing the point 
$y$ contradicting Lemma \ref{propdisc}.
\end{proof}

\begin{lem}\label{cocompact}
The $G$-action on $Y$ is cocompact.
\end{lem}

\begin{proof}
Since $Y$ is a proper metric space, it suffices to prove that there is a closed $d_Y$-metric ball $B_Y$ such that the 
$i(G)$-translates of $B_Y$ cover $Y$.  As the $G$-action on $X$ is cocompact, and $X$ is a proper metric space, 
there exists $x_0 \in X$ and $R>0$ such that the $G$-translates of the closed $d_X$-metric ball 
$$B_X=\{x \in X\, \vert\, d_X(x,x_0)\leq R\}$$ 
cover $M$.  Let 
$$B_Y=\{y \in Y\, \vert\, d_Y(y,\phi x_0)\leq R+3C+L\}.$$  
Fix $y \in Y$.  As the $G$-translates of $B_X$ cover $X$, there exists $g \in G$ such that $\psi y \in g B_X$, 
or equivalently, 
$$d_X(gx_0,\psi y)\leq R.$$  
We conclude the proof by showing that 
$$d_Y(i(g)(\phi x_0),y)\leq R+3C+L$$
 or equivalently that $y \in i(g)B_Y$. Indeed, we can estimate
\begin{eqnarray}
d_Y\big(i(g)(\phi x_0),y\big)&\leq&d_Y\big(i(g)(\phi x_0),\bar{g}(\phi x_0)\big)+d_Y\big(\bar{g}\phi x_0,y\big)\nonumber\\
&\leq&d_Y\big(\bar{g}\phi x_0,y\big)+L\nonumber\\
&=&d_Y\big((\phi g \psi) \phi x_0,y\big)+L\nonumber\\
&\leq&d_X\big(\psi (\phi g\psi \phi x_0),\psi (y)\big)+C+L\nonumber\\
&\leq&d_X\big((\psi \phi)(g\psi \phi x_0), g\psi \phi x_0\big)+d_X(g\psi \phi x_0, \psi y)+C+L\nonumber\\
&\leq&d_X\big(g(\psi \phi x_0), \psi y)+2C+L\nonumber\\
&=&d_X\big(\psi \phi x_0,g^{-1}(\psi y)\big)+2C+L\nonumber\\
&\leq&d_X((\psi \phi) x_0,x_0)+d_X(x_0,g^{-1}\psi y)+2C+L\nonumber\\
&\leq&d_X(x_0,g^{-1}\psi y)+3C+L\nonumber\\
&=&d_X(gx_0,\psi y)+3C+L\nonumber\\
&\leq&R+3C+L.\nonumber
\end{eqnarray}
\end{proof}

Combining Lemma \ref{propdisc}, Corollary \ref{free}, and Lemma \ref{cocompact}, we see that 
the $G$-action on $Y$ given by $i: G\rightarrow \Isom(Y)$ is geometric.


\subsection{An equivariant almost-isometry}

Let $\Omega \subset X$ be a \textit{strict} fundamental domain for the $G$-action on $X$.  In other words, $\Omega$ 
consists of a single point from each $G$-orbit in $X$. Then for each $x \in X$, there exist unique $g \in G$ and 
$\omega \in \Omega$ with $g\omega=x$.

Define $\Phi :X \rightarrow Y$ by $\Phi (x)=\Phi (g\omega) :=i(g)\phi \omega.$  By construction, $\Phi$ is equivariant 
with respect to the $G$ and $i(G)$ actions on $X$ and $Y$ respectively.

\begin{lem}\label{F}
The $(G,i(G))$-equivariant map $\Phi :X \rightarrow Y$ is a $(5C+2L)$-almost-isometry.
\end{lem}

\begin{proof}
Let $x \in X$.  There are unique $g \in G$ and $\omega \in \Omega$ such that $x=g\omega$.  Then 
\begin{eqnarray}
d_Y\big(\phi x,\Phi x\big)&=&d_Y\big(\phi g\omega,\Phi g\omega\big)\nonumber\\
&=&d_Y\big(\phi g\omega,i(g)\phi \omega\big)\nonumber\\
&\leq&d_Y\big(\phi g\omega,\bar{g}\phi \omega\big)+d_Y\big(\bar{g}(\phi \omega),i(g)(\phi \omega)\big)\nonumber\\
&\leq&d_Y\big(\phi g\omega,\bar{g}\phi \omega\big)+L\nonumber\\
&=&d_Y\big(\phi g\omega,(\phi g \psi)\phi \omega\big)+L\nonumber\\
&\leq&d_X\big(g\omega,g\psi \phi \omega\big)+C+L\nonumber\\
&=&d_X\big(\omega,(\psi \phi)\omega\big)+C+L\nonumber\\
&\leq&2C+L.\nonumber
\end{eqnarray}  
It follows that for $x_1,x_2 \in X$
\begin{eqnarray}
d_Y\big(\Phi x_1,\Phi x_2\big)&\leq& d_Y\big(\Phi x_1, \phi x_1\big)+d_Y\big(\phi x_1, \phi x_2\big)+
d_Y\big( \phi x_2, \Phi x_2\big)\nonumber\\
&\leq&d_Y\big( \phi x_1, \phi x_2\big)+4C+2L\nonumber\\
&\leq&d_X\big (x_1,x_2\big )+5C+2L.\nonumber
\end{eqnarray}  
A similar argument gives the estimate
$$d_X\big(x_1,x_2\big) \leq d_Y\big( \Phi x_1, \Phi x_2\big)+5C+2L.$$
The previous two inequalities show that $\Phi $ is a $(5C+2L)$-almost-isometric map.  It remains to show that 
$\Phi$ is $(5C+2L)$-coarsely onto.  Let $y \in Y$.  Then $\psi y \in X$ and 
$$d_Y\big (\Phi (\psi y),y\big)\leq d_Y\big(\Phi (\psi y),\phi (\psi y)\big)+d_Y\big (\phi \psi y,y\big)\leq 2C+L+C,$$ 
concluding the proof.
\end{proof}


\subsection{Comparing the marked length spectrum}

To summarize, we constructed a new $G$-action on $Y$, given by $i:G \rightarrow \Isom(Y)$, which we have
shown to be geometric. We also constructed an {\it equivariant} almost-isometry $\Phi$ from $X$ to $Y$. We now compare
the translation lengths for the $G$-actions on $X$ and $Y$. Let $\bar{C} = 5C+2L$, the almost-isometry constant
for the equivariant almost-isometry $\Phi : X\rightarrow Y$.

\begin{lem}\label{same-MLS}
For every $g\in G$, we have $\tau(g)=\tau(i(g))$. 
\end{lem}

\begin{proof}
By formula (\ref{formula}), for any $x \in X$ we have
\begin{eqnarray}
\tau(g)&=&\lim_{n \rightarrow \infty} \frac{d_X(x,g^n x)}{n}\nonumber\\ \nonumber\\
&\geq& \lim_{n\rightarrow \infty} \frac{d_Y(\Phi x,\Phi g^nx)-\bar{C}}{n}\nonumber\\ \nonumber\\
&=& \lim_{n \rightarrow \infty} \frac{d_Y(\Phi x,i(g)^n \Phi x)}{n}=\tau(i(g)).\nonumber
\end{eqnarray}
An identical argument, using a coarse inverse to $\Phi$, gives the reverse inequality.

\end{proof}


\subsection{Concluding the proof} We now have isometric $G$-actions on $X$ and $Y$.
We have shown that the action on $Y$ is geometric, and that the two actions have the same translation lengths.
Since $X$, by hypothesis, is marked length spectrum rigid, we conclude that there is an equivariant isometric embedding
$\psi: X\rightarrow Y$. Finally, to see that $\psi$ is coarsely onto, we just note that $\partial \psi \equiv \partial \Phi$
(as both these maps are at finite distance from the same orbit map), so $\psi$ and $\Phi$ are at bounded distance
apart. Also, the first part of the proof of Lemma \ref{F} shows $\Phi$ and $\phi$ are at bounded distance apart,
so we deduce that $\psi$ and $\phi$ are at bounded distance apart. Since $\phi$ is coarsely onto, we conclude
$\psi$ is coarsely onto.  This completes the proof of the Theorem \ref{theorem-AI+MLS}.


\subsection{Application -- locally symmetric manifolds.}
In this section we prove Corollary \ref{symmetric}, dealing with quaternionic hyperbolic space 
$\mathbb O \mathbb H^n$ and the Cayley hyperbolic plane $\text{Ca}\mathbb H^2$ .

\begin{proof}[Proof of Corollary \ref{symmetric}]
Pansu \cite{Pa} has shown that $\mathbb O \mathbb H^n$ and $\text{Ca}\mathbb H^2$ are QI-rigid,
and hence AI-rigid. Combining work of Hamenstadt \cite{Ha} and Besson-Courtois-Gallot \cite{BCG}, we also know that 
uniform lattices in the semi-simple Lie groups $Sp(n,1)$ and $F_{4, -20}$ are marked length spectrum rigid within the 
class of actions on negatively curved manifolds of the same dimension as the corresponding symmetric space. 

Following the notation in our Theorem \ref{theorem-AI+MLS}, we let $Y= (\tilde M, \tilde g_0)$ denote the symmetric space,
and $X=(\tilde M, \tilde g_1)$ the universal cover with the exotic metric. Proceeding as in the Main Theorem, we assume
there is an almost-isometry $\phi: X\rightarrow Y$. One then uses AI-rigidity of the symmetric space $Y$ to 
construct a new geometric $G$-action on $Y$, so that that the two $G$-actions
have the same marked length spectrum (Lemma \ref{same-MLS}). Finally, we apply marked length rigidity for the 
$G$-action {\it on the symmetric space $Y$} (rather than on the symmetric space $X$) to obtain a coarsely onto isometric
embedding of $Y$ into $X$. Since $X$, $Y$ are complete Riemannian manifolds of the same dimension, such a map
provides an isometry between $X$ and $Y$. Thus $X$ is also a symmetric space, and so $(M, g_1)$ had to also be 
locally symmetric, as claimed. 
\end{proof}


\subsection{Application -- Fuchsian buildings.}\label{proof-Fuchsian-building}

We start by quickly recalling some of the terminology concerning Fuchsian buildings, which were first introduced by Bourdon
\cite{Bou2}. These are $2$-dimensional polyhedral complexes which satisfy a number of axioms. First, one starts with a 
compact convex hyperbolic polygon $R\subset \mathbb H^2$, with each angle of the form $\pi/m_i$ for some $m_i$
associated to the vertex ($m_i\in \mathbb N,
m_i\geq 2$). Reflection in the geodesics extending the sides of $R$ generate a Coxeter group $W$, and the orbit of $R$ under
$W$ gives a tessellation of $\mathbb H^2$. Cyclically labeling the vertices of $R$ by the integers $\{1\}, \ldots, \{k\}$ (so that the
$i^{th}$ vertex has angle $\pi/m_i$), and
the corresponding edges by $\{1,2\}, \{2,3\}, \ldots ,\{k, 1\}$, one can apply the $W$ action to obtain a $W$-invariant labeling
of the tessellation of $\mathbb H^2$; this labeled polyhedral $2$-complex will be denoted $A_R$, and called the {\it model 
apartment}.

A polygonal $2$-complex $X$ is called a $2$-dimensional hyperbolic building if it contains a vertex labeling by the integers
$\{1, \ldots , k\}$, along with a distinguished collection of subcomplexes $\mathcal A$ called the {\it apartments}. The individual
polygons in $X$ will be called {\it chambers}. The complex is required to have the following properties:
\begin{itemize}
\item each apartment $A\in \mathcal A$ is isomorphic, as a labeled polygonal complex, to the model apartment $A_R$,
\item given any two chambers in $X$, one can find an apartment $A\in \mathcal A$ which contains the two chambers, and
\item given any two apartments $A_1, A_2\in \mathcal A$ that share a chamber, there is an isomorphism of labeled 
$2$-complexes $\phi: A_1\rightarrow A_2$ that fixes $A_1\cap A_2$. 
\end{itemize}
If in addition each edge labeled $i$ has a fixed number $q_i$ of incident polygons, then $X$ is called a {\it Fuchsian building}.
For a Fuchsian building, the combinatorial axioms force some additional structure on the links of vertices: these graphs 
must be
{\it generalized $m$-gons} in the sense of Tits. Work of Feit and Higman \cite{FH} then implies that each $m_i$ must lie in the set $\{2, 3, 4, 6, 8\}$.
Note that making each polygon in $X$ isometric to $R$ via the label-preserving map produces a CAT(-1) metric on $X$.
However, a given polygonal $2$-complex might have several metrizations as a Fuchsian building: these correspond to 
varying the hyperbolic metric on $R$ while preserving the angles at the vertices. Any such variation induces a new 
CAT(-1) metric on $X$. The hyperbolic polygon $R$ is called {\it normal} if it has an inscribed circle that touches all its sides --
fixing the angles of a polygon to be $\{\pi/m_1, \ldots ,\pi/m_k\}$, there is a unique normal hyperbolic polygon with those 
given vertex angles.  We can now state a rigidity result for Fuchsian buildings.


\begin{cor}\label{Fuchsian-building}
Let $G$ be a group acting freely and cocompactly on a combinatorial Fuchsian buildings $X$  having no vertex links which are generalized $3$-gons. Let $d_0$ be the metric on $X/\Gamma$ where each chamber 
is the normal hyperbolic polygon, and let $d_1$ be a locally CAT(-1) metric, where each polygon has a Riemannian 
metric of curvature $\leq 1$ with geodesic sides. 
Then the universal covers $(X, \tilde d_0)$ and $(X, \tilde d_1)$ are almost-isometric 
if and only if they are isometric, in which case the isometry can be chosen to be equivariant with respect to the $G$-actions,
and hence $(X/\Gamma, d_0)$ is isometric to $(X/\Gamma, d_1)$.
\end{cor}

\begin{proof}[Proof of Corollary \ref{Fuchsian-building}] The argument for this is similar to the proof of Corollary \ref{symmetric}.
Let $\phi: (X, \tilde d_0) \rightarrow (X, \tilde d_1)$ be the almost-isometry between the universal covers. For the Fuchsian
building $(X, \tilde d_0)$, Xie \cite{X} has established QI-rigidity (and hence AI-rigidity). It is important here that 
for the $\tilde d_0$-metric all polygons are normal -- otherwise QI-rigidity does {\bf not} hold. Using the AI-rigidity, we 
can construct a new geometric $\Gamma$-action on $(X, \tilde d_0)$. The $\Gamma$-actions on $(X, \tilde d_0)$ and
$(X, \tilde d_1)$ now have the same marked length spectrum (see Lemma \ref{same-MLS}). But Constantine and Lafont
\cite{CL} have established that, when there are no vertex links which are generalized $3$-gons, the metric $\tilde d_0$ is 
marked length spectrum rigid within the class of metrics described in the statement of our corollary (thus including $\tilde d_1$). 
This establishes the corollary.
\end{proof}


\section{AIs and volume growth}\label{section-proof-of-volgrowth}

In this section, we establish Theorem \ref{theorem-volgrowth}. We start by reminding
the reader of
a standard packing/covering argument, which allows us to reinterpret volume growth entropy in 
terms of quantities we can estimate.

\begin{lem}\label{count-covering}
Let $M$ be a Riemannian cover of a compact manifold.  Fix a basepoint
$p\in M$, a parameter $s>0$, and define the counting function $N(s,r)$ to be the minimal cardinality of a 
covering of $B_p(r)$ by balls of radius $s$. Then for any choice of $s$, we have that 
$$h_{vol}(M) = \lim _{r\to \infty} \frac{\ln \left(N(s,r)\right)}{r}.$$
\end{lem}

\begin{proof}
Let $V_{s}<\infty$ be the maximal volume of a ball of radius $s$, and $v_{s}>0$ be the minimal
volume of a ball of radius $s/2$ (so clearly $v_{s}< V_{s}$).  A maximal packing of $B_p(r)$ by disjoint balls of radius $s/2$ induces a covering of $B_p(r)$ by balls of radius $s$ with the same centers.  We thus obtain the following bounds:
$$\frac{Vol\left( B_p(r)\right)}{V_{s}}\leq N(s,r) \leq \frac{Vol\left( B_p(r)\right)}{v_{s}}.$$
Since both $V_s$, $v_{s}$ are fixed real numbers, taking the log and the limit as $r\to \infty$ yields the Lemma.
\end{proof}

Now with Lemma \ref{count-covering} in hand, the proof is straightforward. We will use the almost-isometry 
to relate the counting function $N_1(s, r)$ for the manifold $M_1$ to the counting function $N_2(s', r')$ for the 
manifold $M_2$.

Let $\phi: M_1 \rightarrow M_2$
be the $C$-almost-isometry. Choose a basepoint $p\in M_1$, and let $q=\phi(p)$ 
be the basepoint in $M_2$. Consider the counting function $N_1(1,r)$ for the 
manifold $M_1$. For a given $r$, let $\{p_1, \ldots , p_N\}$ (where $N:=N_1(1, r)$) be the centers of the balls of 
radius $1$ for the minimal covering
of $B_p(r)$, and let $q_i:= \phi(p_i)$ be the corresponding image points in $M_2$. 

The covering of $B_p(r)$ by the set of balls $\{B_{p_i}(1)\}_{i=1}^N$ maps over to a covering 
$\{\phi \left( B_{p_i}(1)\right)\}_{i=1}^N$ of the set $\phi\left(B_p(r)\right)$. Since $\phi$ is an almost-isometry 
with additive constant $C$, we have for each $i$ that
$$\phi \left( B_{p_i}(1)\right) \subseteq B_{q_i}(1+C),$$
and hence we also have a covering $\{B_{q_i}(1+C)\}_{i=1}^N$ of the set $\phi\left(B_p(r)\right)$ 
by metric balls centered at $\{q_1, \ldots ,q_N\}$. 

Next, we note that the $C$-neighborhood of the set $\phi\left(B_p(r)\right)$ contains the set $B_q(r-2C)$.
Indeed, we know that $\phi(M_1)$ is $C$-dense in $M_2$, so given an arbitary point $x\in B_q(r-2C)$, we
can find a point $y\in M_1$ with the property that $d_2(\phi y , x) < C$. Now assume $y$ lies outside of $B_p(r)$.
Then $d_1(y, p)>r$, which would imply 
$$d_2(\phi y, q)= d_2(\phi y, \phi p) \geq d_1(y, p) -C > r-C.$$
Since $d_2(\phi y , x) < C$, the triangle inequality forces $d_2(x, q)>r-2C$, a contradiction. So we must have
$y\in B_p(r)$. 

Since the $C$-neighborhood of $\phi\left(B_p(r)\right)$ contains the set $B_q(r-2C)$, and we have a covering
$\{B_{q_i}(1+C)\}_{i=1}^N$ of the set $\phi\left(B_p(r)\right)$ by metric balls, we obtain a corresponding covering 
$\{B_{q_i}(1+2C)\}_{i=1}^N$ of the set $B_q(r-2C)$ by balls of radius $1+2C$. This implies that
$$N_1(1, r) \geq N_2(1+2C, r-2C).$$
Taking the log and the limit as $r\to \infty$, and taking into account Lemma \ref{count-covering}, we obtain the
pair of inequalities:
$$h_{vol}^+(M_1) \geq h_{vol}^+(M_2) \hskip 0.5in h_{vol}^-(M_1) \geq h_{vol}^-(M_2).$$
Applying the same argument to a coarse inverse almost-isometry
yields the pair of reverse inequalities, 
completing the proof of Theorem \ref{theorem-volgrowth}.

\begin{rem}
In the special case where the $M_i$ both have metrics of bounded negative sectional curvature, and support compact quotients, one can give an alternate proof of Theorem \ref{theorem-volgrowth} by exploiting the metric structures on the boundaries at infinity.
Indeed, fixing a basepoint $p\in M_1$ and corresponding basepoint $q:=\phi(p)$, one can construct metrics
on the boundaries at infinity $\partial _\infty M_1$ and $\partial _\infty M_2$. It follows then from work of 
Bonk and Schramm that the almost-isometry $\phi: M_1\rightarrow M_2$ induces a bi-Lipschitz homeomorphism
$\phi_\infty: \partial _\infty M_1 \rightarrow \partial _\infty M_2$ (see \cite[proof of Theorem 6.5]{BS}). 
In particular, the two boundaries have
identical Hausdorff dimension. But Otal and Peign\'e \cite{OP} have shown that for such manifolds, the Hausdorff dimension of the boundary at infinity coincides with the 
topological entropy of the geodesic flow on the compact quotient of the $M_i$ 
(which by Manning \cite{Ma} coincides with the volume growth entropy of the $M_i$).
\end{rem}

\subsection{Application - rigidity results.}
We now give a proof of Corollary \ref{rigidity-results}.

\begin{proof} We deal with each of the various cases separately.

\vskip 10pt

\noindent \underline{Case (1):} The manifold $M$ is finitely covered by the $2$-torus $T^2$.
Lifting the metrics $g_0, g_1$ to this finite cover, we see that it is enough to deal with the case where $M=T^2$. Then the 
metrics $\tilde g_0, \tilde g_1$ can be viewed as a pair of $\mathbb Z^2$-invariant metrics on $\mathbb R^2$. Associated to
these two periodic metrics, we have a pair of Banach norms on $\mathbb R^2$ defined via:
$$||v||_i := \lim_{r\to \infty} \frac{d_i(0, r v)}{r}$$
where $d_i$ is the distance function associated to the metric $g_i$. Burago \cite{Bu} 
showed that the identity map on $\mathbb R^2$
provides an almost-isometry from the Banach norm to the original periodic metric, i.e. there is a constant $C$ with the property
that for all vectors $v, w \in \mathbb R^n$, we have:
$$\left| ||v-w||_i - d_i(v, w)\right| < C.$$
We note that there is an alternate way to view the Banach norm: consider the pointed space $(\mathbb R^2, 0)$ 
with the sequence of metrics given by $\frac{d_i}{n}$ ($n\in \mathbb N$), and take the ultralimit. The resulting pointed space, the \textit{asymptotic cone}, is 
topologically $(\mathbb R^2, 0)$, equipped with the corresponding Banach norm (regardless of the choice of ultrafilter). We
denote by $F_i$ the unit ball, centered at $0$, in the Banach norm $||\cdot ||_i$.

Now assume we have an almost-isometry $\phi: (\mathbb R^2, d_0) \rightarrow (\mathbb R^2, d_1)$. Then passing to the
asymptotic cones, we obtain an {\it isometry} $\hat \phi: (\mathbb R^2, ||\cdot ||_0) \rightarrow (\mathbb R^2, ||\cdot||_1)$ 
fixing $0$, and sending the unit ball $F_0$ to the unit ball $F_1$. Since the geodesics in any Banach norm are straight lines, 
the map $\hat \phi$ is a linear map. Now for the flat metric $\tilde g_0$, we know that the associated Banach norm is a Euclidean
norm (i.e. the unit ball $F_0$ is an ellipsoid). Since $\hat \phi$ is linear, we have that $\hat \phi(F_0) = F_1$ is also an ellipsoid, and hence that $||\cdot ||_1$ is a (smooth) Euclidean norm. 

By Bangert's \cite[Theorem 5.3]{Ba}, the periodic minimal geodesics of $(T^2,g_1)$ in any nontrivial free homotopy class of $T^2$ foliate $T^2$.  By Innami \cite{In} (or \cite[proof of Theorem 6.1]{Ba}), the metric $g_1$ must also be flat.



\vskip 10pt

\noindent \underline{Cases (2-4):}
By our Theorem \ref{theorem-volgrowth} we have 
$h_{vol}(\tilde g_0) = h_{vol}(\tilde g_1)$ which immediately implies that $h_{vol}(\tilde g_0)\cdot Vol(g_0)\geq h_{vol}(\tilde g_1)\cdot Vol(g_1).$  Locally symmetric metrics uniquely minimize the functional $h_{vol}(-)^n\cdot Vol(-)$ in case (2) by Katok \cite{Ka}, in case (3) by Besson, Courtois, and Gallot \cite{BCG}, and in the conformal class in case (4) by Knieper \cite{Kn}.  In each of these cases, we conclude that $\tilde g_1 = \lambda \tilde g_0$ for some $0<\lambda<\infty$. Corollary
\ref{no-scale} implies $\lambda =1$, completing the proof of Cases (2-4).

\vskip 10pt

\noindent \underline{Case (5):} Let us briefly specify the metric $g_0$ -- for this metric, the individual negatively curved 
symmetric spaces factors are scaled as in \cite[Section 2]{CF}. Connell and Farb have now shown that the metric $g_0$
is the unique minimizer for the volume growth entropy on the space of locally symmetric metrics on $M$. In 
\cite[Theorem A]{CF}, they then proceed to show that $g_0$ is the unique minimizer of the functional $h_{vol}(-)^n\cdot Vol(-)$
on the space of all metrics on $M$. The same
argument as in cases (2-4) give the desired conclusion.

\end{proof}


\subsection{Application - the case of metric trees.} While we have primarily focused on Riemannian manifolds, some of
our results hold in greater generality. For instance, the proof of Theorem \ref{theorem-volgrowth} did not make any 
particular use of the fact that our metric was Riemannian. In fact, the very same proof yields the following more general
result. For $(X,d)$ a metric space of Hausdorff dimension $s$, denote by $\mathcal H^s$ the 
$s$-dimensional Hausdorff measure, and define the upper/lower exponential volume growth rate to be
$$h^+(X,d) := \limsup _{r\to \infty} \frac{ \ln \left( \mathcal H^s\left( B_p(r) \right) \right)}{r} \hskip 0.5in 
h^-(X,d) := \liminf _{r\to \infty} \frac{ \ln \left( \mathcal H^s\left( B_p(r) \right) \right)}{r} $$
where $B_p(r)$ is the metric ball of radius $r$ centered at a fixed basepoint $p\in X$ (these are independent
of the choice of basepoint). In the case where $h^+(X,d)=h^-(X,d)$, we denote the common value by $h(X,d)$, which
we call the exponential volume growth rate of $X$. The proof of Theorem \ref{theorem-volgrowth} in fact establishes: 

\begin{thm}\label{general-volume-growth}
Let $(X, d_1)$, $(X, d_2)$ be a pair of metric spaces of Hausdorff dimension $s$, and assume that there are two sided 
bounds on the $s$-dimensional Hausdorff measure of balls of any given radius. Then if $(X, d_1)$ is almost isometric
to $(X, d_2)$, we must have $h^+(X, d_1) = h^+(X,d_2)$, and $h^-(X, d_1) = h^-(X,d_2)$. 
\end{thm}


For an easy example illustrating this more general setting, consider the setting of connected metric graphs. 
The $1$-dimensional Hausdorff measure of a $\tilde d$-ball of radius $r$ in the graph will then be the sum of 
the edge lengths of the (portions of) edges inside the ball. If one 
imposes a lower bound on the length of edges, and an upper bound on the degree of vertices, 
this easily leads to two 
sided bounds on the $1$-dimensional Hausdorff measure of balls of any given radius. So Theorem \ref{general-volume-growth}
applies to this class of metric spaces.

Let us give an application of this: consider a finite combinatorial graph $X$, with the property
that each vertex has degree $\geq 3$. The universal cover of $X$ is then a combinatorial tree $T$. One can metrize $X$
in many different ways, by assigning lengths to each edge, and making each edge isometric to an interval of the corresponding
length. We let $\mathcal M(X)$ be the space of such metrics.
Any such metric $d$ lifts to give a $\pi_1(X)$-invariant metric $\tilde d$ on the tree $T$, with lower bounds on the edge
lengths and upper bounds on the degree of vertices. In this special case one has that $h^+(T, \tilde d)=h^-(T, \tilde d)$,
and we will denote the common value by $h_{vol}(d)$. Then Theorem \ref{general-volume-growth} 
tells us that for $d_0, d_1\in \mathcal M (X)$ arbitrary, if $(T, \tilde d_0)$ is almost-isometric to $(T, \tilde d_1)$, then 
$h_{vol}(d_0) = h_{vol}( d_1)$.

We now view $h_{vol}$ as a function on the space $\mathcal M(X)$, an open cone inside some large $\mathbb R^n$ 
(where $n$ is the number of edges in $X$). It is easy to see, from the scaling property of Hausdorff
dimension, that 
$$h_{vol}(\alpha \cdot d) = \frac{1}{\alpha} h_{vol}(d).$$
As such, it is reasonable to impose a normalizing condition, e.g. letting $\mathcal M_1(X) \subset \mathcal M(X)$ 
be the subspace of metrics whose sum of lengths is $=1$. The behavior of $h_{vol}$ on the subspace $\mathcal M_1(X)$
was studied by Lim in her thesis, and she showed \cite{Li} that there is a {\it unique} metric $d_0$ which {\it minimizes}
$h_{vol}$ -- moreover, she gave an explicit computation of this metric in terms of the degrees at the various vertices of $X$.
Some related work was done by Kapovich and Nagnibeda \cite{KS} and by Rivin \cite{Ri}. In conjunction with Lim's 
result, our Theorem \ref{general-volume-growth} implies the following:

\begin{cor}\label{rigidity-tree}
Let $X$ be a combinatorial graph, $d_0$ the metric produced by Lim, and $d_1\in \mathcal M_1(X)$ any metric on
$X$ distinct from $d_0$. Then $(T, \tilde d_0)$ and $(T, \tilde d_1)$ are {\bf not} almost isometric.
\end{cor}

\begin{rem}
The reader will undoubtedly wonder as to whether some similar result holds for the Fuchsian buildings discussed in 
Section \ref{proof-Fuchsian-building}. While Theorem \ref{general-volume-growth} applies to Fuchsian 
buildings (of course, using $2$-dimensional Hausdorff measure, and appropriate constraints on the metrics),
the behavior of the functional $h_{vol}$ on the
corresponding moduli space of metrics is much more mysterious. In particular, (local) minimizers of the functional are
not known, and indeed uniqueness of such a minimizer is not known (see Ledrappier and Lim \cite{LL} for some
work on this question).
\end{rem}



\section{Concluding remarks}\label{section-concluding-remarks}

Much of the work in this paper was motivated by the following:

\vskip 10pt

\noindent {\bf Question 1:} Let $M$ be an aspherical manifold with universal cover $\tilde M$. Can one
find a pair of Riemannian metrics $g, h$ on $M$, whose lifts to the universal cover $(\tilde M, \tilde g), 
(\tilde M, \tilde h)$ are almost-isometric but {\bf not} isometric?

\vskip 10pt



Our results in this paper give a number of examples (see Corollaries \ref{symmetric}, \ref{entropy}, \ref{rigidity-results})
of pairs of metrics on compact manifolds whose lifts to the universal 
cover are quasi-isometric, but {\bf not} almost-isometric. Thus any QI between the universal covers must have multiplicative
constant $>1$. One can ask whether there is a ``gap'' in the multiplicative constant. We suspect this is not the case in 
general.

\vskip 10pt

\noindent {\bf Question 2:} Can one find an aspherical manifold $M$ and a pair of Riemannian metrics $g, h$, with the 
property that
the universal covers are $(C_i, K_i)$-quasi-isometric via a sequence of maps $f_i$, where $C_i\to 1$, but are {\bf not}
almost-isometric.

\vskip 10pt

In the special case where $M$ is a higher genus surface, and the metrics under consideration are negatively curved,
one has complete answers to both of the above questions (see \cite{LSvL}).

\vskip 5pt

In a different direction, we saw in our Theorem \ref{theorem-volgrowth} that the rate of exponential growth is an
almost-isometry invariant (though it is {\bf not} a quasi-isometry invariant). At the other extreme, universal covers of
infra-nil manifolds, equipped with the lift of a metric, are known to have {\it polynomial} growth. More precisely, 
$Vol(B(r)) \sim C(g)\cdot r^k$ where the integer $k\in \mathbb N$ depends only on $M$, but the constant $C(g)$ depends 
on the chosen metric $g$ on $M$. One can ask the following:

\vskip 10pt

\noindent {\bf Question 3:} Let $M$ be an infra-nil manifold, and $g, h$ a pair of Riemannian metrics on $M$. Denote by $C(g),
C(h) \in (0, \infty)$ the coefficient for the polynomial growth rate of balls in $\tilde M$. If $(\tilde M, \tilde g)$ is almost-isometric
to $(\tilde M, \tilde h)$, does it follow that $C(g) = C(h)$?

\vskip 10pt

It is easy to see that the estimates appearing in our proof of Theorem \ref{theorem-volgrowth} are too crude to deal
with the coefficient of polynomial growth. In the special case where $(M,g)$ is a flat surface, we have an affirmative answer to 
{\bf Question 3}: our Corollary \ref{rigidity-results} implies that $h$ must also be flat, from which it is immediate that 
$C(g)=C(h)$. Observe that the exponential volume growth rate can alternatively be interpreted as an isoperimetric
profile, or as a filling invariant (in the sense of Brady and Farb \cite{BF}). One could also ask whether one can use
these alternate viewpoints to define some new almost-isometry invariants.

\vskip 10pt 

We have focused on almost-isometric metrics on the universal cover of a fixed topological manifold $M$.
We could also ask similar questions for a pair of closed smooth manifolds $(N^n, g)$, $(M^m, h)$ where $n\leq m$. For 
instance, can one find an almost-isometric embedding $(\tilde N, \tilde g) \rightarrow (\tilde M, \tilde h)$ which is not at finite 
distance from an isometric embedding? In the case where the universal covers are isometric to irreducible (Euclidean) 
buildings, or to irreducible non-positively curved symmetric spaces of equal rank $r>1$, recent work of Fisher and Whyte 
establishes that every almost-isometric embedding is at finite Hausdorff distance from an isometric embedding
(see \cite[Corollary 1.8]{FW}).

\vskip 10pt

Finally, while our purpose in this paper was mostly the study of {\it spaces} up to almost-isometry, one can ask similar
questions at the level of finitely generated {\it groups}. One says that a pair of finitely generated groups $G, H$ are 
almost isometric provided one can find finite symmetric generating
sets $S\subset G$, $T\subset H$ so that the corresponding metrics spaces $(G, d_S)$ and $(H, d_T)$ are almost-isometric.
A basic problem here is to resolve:

\vskip 10pt

\noindent {\bf Question 4:} Let $G$, $H$ be a pair of quasi-isometric groups. Must they be almost-isometric?

\vskip 10pt

For instance, it is easy to see that commensurable groups are almost-isometric. In general, one suspects that the answer
should be ``no'', though again examples seem elusive. The corresponding question for bi-Lipschitz equivalence was answered
in the negative by Dymarz \cite{Dy}. One aspect which seems to make the almost-isometric question harder than the
corresponding bi-Lipschitz question lies in the fact that distinct word metrics on a fixed finitely generated group $G$ are 
not a priori almost-isometric to each other, whereas they are always bi-Lipschitz equivalent. This means that understanding
groups up to AI involves understanding {\it all} word metrics. For instance, let us specialize to the case where the groups
$G, H$ have exponential growth. Then by varying the possible generating sets for $G, H$, and looking at the 
corresponding exponential volume growth rate, one obtains the {\it growth spectra} $\spec (G), \spec (H) \subset (0, \infty)$.
An affirmative answer to {\bf Question 4} would imply, by Theorem \ref{general-volume-growth},
that when $G$, $H$ are quasi-isometric, $\spec(G)\cap \spec(H)\neq \emptyset$ -- a result which seems unlikely to be true in full generality.

\vskip 20pt

\centerline{\bf Acknowledgments}

\vskip 10pt

The authors thank Xiangdong Xie for helpful comments. 

\vskip 20pt




\end{document}